\documentclass[12pt,oneside,reqno]{amsart}
\usepackage{jrmacros}
\usepackage[margin=1.2in]{geometry}

\theoremstyle{plain}
\newtheorem{theorem}{Theorem}[section]
\newtheorem{proposition}[theorem]{Proposition}

\newtheorem{lemma}[theorem]{Lemma}

\newtheorem{conj}{Conjecture}
\newtheorem*{conj1'}{Conjecture 1$'$}

\theoremstyle{definition}
\newtheorem{definition}[theorem]{Definition}

\newtheorem{ques}{Question}
\newtheorem*{ques4'}{Question 4$'$}

\theoremstyle{remark}

\def\as{\hat{\cA}}

\def\bs{\mathbf{s}}
\def\bt{\mathbf{t}}
\def\od{\overline{d}}
\def\mzvs{MZVs}

\def\zt{\hat{\zeta}}
\def\z{\zt}

\def\J#1{\ker(\zt_{#1})}
\def\hhqs{\hh^1_*}
\def\Ji{\ker(\z)}
\def\sg{\hspace*{3mm}}
\def\CR{\Phi_1} \def\DR{\psi_1}
\def\CD{\Phi_2}  \def\DD{\psi_2}

\def\CCC{\Phi} \def\DDD{\psi}
\def\bs{\mathbf{s}}
\def\p{\varphi}

\begin{document}
\title{Asymptotic relations for\linebreak weighted finite multiple zeta values}
\author{Julian Rosen}

\begin{abstract}
Multiple zeta values are real numbers defined by an infinite series generalizing values of the Riemann zeta function at positive integers. Finite truncations of this series are called multiple harmonic sums and are known to have interesting arithmetic properties. When the truncation point is one less than a prime $p$, the mod $p$ values of multiple harmonic sums are called finite multiple zeta values. The present work introduces a new class of congruence for multiple harmonic sums, which we call \emph{weighted congruences}. These congruences can hold modulo arbitrarily large powers of $p$. Unlike results for finite multiple zeta values, weighted congruences typically involve harmonic sums of multiple weights, which are multiplied by explicit powers of $p$ depending on weight. We also introduce certain formal weighted congruences inolving an infinite number of terms, which we call \emph{asymptotic relations}. We define a weighted analogue of the finite multiple zeta function, and give an algebraic framework for classifying weighted congruences and asymptotic relations.
\end{abstract}
\maketitle

%
%
\section{Introduction}
The Riemann zeta function is defined by the infinite series
\[
\zeta(s):=\sum_{n\geq 1}\frac{1}{n^s},
\]
which converges provided that $\re(s)>1$. Values of this function are particularly interesting when $s\geq 2$ is an integer.
In 1735 Euler computed
\[
\zeta(2)=\frac{\pi^2}{6}.
\]
Euler's computation actually gave all of the even zeta values in terms of Bernoulli numbers, which are rational numbers defined by a generating function:
\[
\zeta(2k)=\frac{(-1)^{k-1}2^{2k-1}B_{2k}}{(2k)!}\pi^{2k},\gap \gap\sum_{n\geq 0}B_n\frac{x^n}{n!}=\frac{x}{e^x-1}.
\]
In particular $\zeta(2k)$ is a rational multiple of $\pi^{2k}$, so lies in $\Q[\zeta(2)]$. Lindemann proved in 1882 that $\pi$ is transcendental, so the algebraic structure of $\zeta(2),\zeta(4),\ldots$ is completely determined.

Much less is known about the odd zeta values. Ap\'{e}ry showed in 1978 that $\zeta(3)$ is irrational, and in 2001 Rivoal \cite{Riv00} proved that infinitely many of the odd zeta values are irrational. Zudilin \cite{Zud01} showed in 2001 that at least one of the four numbers $\zeta(5)$, $\zeta(7)$, $\zeta(9)$, or $\zeta(11)$ is irrational.  It is conjectured that $\zeta(3), \zeta(5), \ldots$ are algebraically independent over $\Q(\pi)$, but at present this conjecture is inaccessible.

%
%
\subsection{Multiple zeta values}
For $k\geq 1$, the \emph{multiple zeta function} of depth $k$ is given by
\[
\zeta(s_1,\ldots,s_k):=\sum_{n_1>\ldots>n_k\geq 1}\frac{1}{n_1^{s_1}\ldots n_k^{s_k}}\in\R.
\]
For $k=1$ this is just the Riemann zeta function.
Values of the multiple zeta function for $s_1,\ldots,s_k$ integers are called \emph{multiple zeta values} (\emph{\mzvs}) or \emph{Euler-Zagier sums}. To ensure convergence, we insist $s_1\geq 2$ and $s_2,\ldots,s_k\geq 1$.

\mzvs{} satisfy many algebraic relations . For example, $\zeta(2,1)=\zeta(3)$, a fact that was known to Euler.
Another family of examples is given by the sum formula, which says that for all positive integers $k<n$,
\[
\sum_{\substack{s_1+\ldots+s_k=n\\s_1\geq 2,s_2,\ldots,s_k\geq 1}}\zeta(s_1,\ldots,s_k)=\zeta(n).
\]
This relation was conjectured by Moen (see \cite{Hof92}) and proved independently by Zagier and Granville \cite{Gra97}.
\begin{ques}
\label{quesmzv}
Determine the set of relations satisfied by \mzvs.
\end{ques}
It can be shown by expanding out a product of nested sums that the product of two \mzvs{} is a linear combination (with integer coefficients) of \mzvs.
This is called a quasi-shuffle relation, and it shows that to answer Question \ref{quesmzv} it suffices to consider \emph{linear} relations.  Also, both of the relations given above are homogeneous (that is, the value of $s_1+\ldots+s_k$ is constant among all terms $\zeta(s_1,\ldots,s_k)$ appearing in each relation). It is conjectured that the space spanned by \mzvs{} is graded by weight, mening that every relation among \mzvs{} can be decomposed into homogeneous relations.
A general method for generating relations among \mzvs{} is given by the extended double shuffle relations, described briefly in Sec.\ \ref{ssrel}. It is conjectured that every relation is a consequence of the extended double shuffle relations (see \cite{Iha06}, Conjecture 1).


%
%
\subsection{Multiple harmonic sums}
\label{subsecMHS}
Truncations of the multiple zeta value series are called multiple harmonic sums, and they have interesting arithmetic properties.
Recall that a \emph{composition} is a finite ordered list $\bs=(s_1,\ldots,s_k)$ of positive integers. We define the \emph{weight} and \emph{depth} (or \emph{length}) of a composition $\bs$ by $w(\bs)=s_1+\ldots+s_k$ and $\ell(\bs)=k$, respectively. We allow the empty composition $\varnothing$.
\begin{definition}
Let $\bs=(s_1,\ldots,s_k)$ be a composition, $n$ a positive integer. We define the \emph{multiple harmonic sum}
\[
H_n(\bs)=H_n(s_1,\ldots,s_k):=\sum_{n\geq n_1>\ldots>n_k\geq 1}\frac{1}{n_1^{s_1}\ldots n_k^{s_k}}\in\Q.
\]
By convention we take $H_n(\varnothing)=1$.
\end{definition}
Multiple harmonic sums are known to satisfy many congruences when $n=p-1$, $p$ a prime. For example, Wolstenholme's Theorem is the old result that the congruence\footnote{For $r$ a rational number and $p$ a prime, we write $r\equiv 0\mod p^n$ to mean $p^n$ divides the numerator of $r$ (i.e., $r\equiv 0\mod p^n$ means $r\in p^n\Z_{(p)}$)} 
\begin{equation}
\label{eqwolst}
\H(1)\equiv0\mod p^2
\end{equation}
 holds for $p\geq 5$. As in this example, it is typical for the congruences we consider to fail for a finite number of small primes.

Systematic study of the mod $p$ structure of multiple harmonic sums was undertaken in independent works of Hoffman \cite{Hof04a} and Zhao \cite{Zha08}. The current formulation of the problem, due to Zagier, uses the ring
\begin{equation*}\label{eqcA}
\cA:=\frac{\prod_p \Z/(p)}{\bigoplus_p \Z/(p)},
\end{equation*}
which is an algebra over $\Q$. An element of $\cA$ consists of a family of residue classes $a_p\in\Z/(p)$ defined for all but finitely many $p$, where two families $(a_p)$ and $(b_p)$ are identified if and only if $\{p:a_p\neq b_p\}$ is a finite set.
\begin{definition}
The \emph{finite multiple zeta value} (or \emph{finite MZV}) is given by
\[
\zeta_\cA(s_1,\ldots,s_k):=\left[\H(s_1,\ldots,s_k)+p\Z\right]\in\cA.
\]
\end{definition}
Mod $p$ congruences for multiple harmonic sums $\H(\bs)$ holding for all sufficiently large $p$ correspond to relations satisfied by the finite \mzvs. Wolstenholme's congruence \eqref{eqwolst}, for instance, implies the vanishing of $\zeta_{\cA}(1)$.

\begin{ques}
\label{quesMHS}
Determine the set of relations satisfied by the finite \mzvs.
\end{ques}
\noindent As with the ordinary \mzvs, it suffices to consider linear relations. There is also a conjecture that the space of finite \mzvs{} is graded by weight.

%
%
\subsection{Weighted congruences and asymptotic relations}
The present work introduces two new, related classes of identities for multiple harmonic sums.

\subsubsection{Weighted congruences} Wolstenholme's congruence \eqref{eqwolst} can be extended. For all $p\geq 5$ we have the congruence\footnote{This can be shown using known expressions for $\H(1)$ and $\H(2,1)$ in terms of Bernoulli numbers; the specific form of the expressions can be found in \cite{Zha08}}
\begin{equation}
\label{eqexs}
\H(1)+\frac{1}{3}p^2\H(2,1)\equiv 0\mod p^4,
\end{equation}
from which \eqref{eqwolst} follows by reducing modulo $p^2$. We will extend this congruence further, adding an additional multiple harmonic sum to obtain
\begin{equation}
\label{eqex}
\H(1)+\frac{1}{3}p^2\H(2,1)-\frac{1}{6}p^4\H(4,1)\equiv 0\mod p^5,
\end{equation}
which holds for all $p\geq 5$. This is proven in Sec.\ \ref{sswolst}. This congruence differs from results for finite \mzvs{} in that each multiple harmonic sum is multiplied by an explicit power of $p$ (we call this a \emph{weighting}). The following definition is new:

\begin{definition}
A \emph{weighted congruence for multiple harmonic sums} (or simply a \emph{weighted congruence}) is a congruence of the form
\begin{equation}
\label{eqmain1}
\sum_{w(\bs)<n}\alpha_{\bs}\,p^{w(\bs)}\H(\bs)\equiv0\mod p^n
\end{equation}
which holds for all sufficiently large $p$, where $n$ is a positive integer, the sum is over compositions $\bs$ of weight less than $n$, and the coefficients $\alpha_{\bs}\in\Q$ do not depend on $p$.
\end{definition}

Any congruence that can be put in the form \eqref{eqmain1} through multiplication by a power of $p$ will also be called weighted. For instance, the congruence \eqref{eqex} is weighted because multiplication by $p$ puts it in the form \eqref{eqmain1} for $n=6$, with $\alpha_{(1)}=1$, $\alpha_{(2,1)}=\frac{1}{3}$, $\alpha_{(4,1)}=-\frac{1}{6}$, and all other $\alpha_{\bs}$ zero.
We pose the following variation of Question \ref{quesMHS}:
\begin{ques}
\label{quesMHS'}
Determine the set of weighted congruences.
\end{ques}
The weighted congruences include all \emph{homogeneous} relations among the finite multiple zeta values, i.e.\ all congruences of the form
\[
\sum_{w(\bs)=n}\alpha_{\bs} \,\H(\bs)\equiv0\mod p\gap\text{ for $p$ large}.
\]
Therefore if we assume the conjecture that the space of finite \mzvs{} is graded by weight, a solution to Question \ref{quesMHS'} would give a solution to Question \ref{quesMHS}. Unlike what is conjectured for finite \mzvs, weighted congruences can be essentially non-homogeneous (as \eqref{eqex} illustrates).

%
%
\subsubsection{Asymptotic relations}
One often finds that a weighted congruence can be extended to stronger one, as we saw above with \eqref{eqwolst} extended to \eqref{eqexs}, and then to \eqref{eqex}. In this example and in many others, extensions can be continued indefinitely. The following definition is new:
\begin{definition}
An \emph{asymptotic relation for weighted multiple harmonic sums} (or simply an \emph{asymptotic relation}) is a formal infinite sum
\begin{equation}
\label{eqar}
\sum_{\bs}\alpha_{\bs}\,p^{w(\bs)}\H(\bs),
\end{equation}
with coefficients $\alpha_{\bs}\in\Q$ not depending on $p$, such that for every positive integer $n$, the weighted congruence
\begin{equation}
\label{eqtrunc}
\sum_{w(\bs)<n}\alpha_{\bs}\,p^{w(\bs)}\H(\bs)\equiv 0\mod p^n
\end{equation}
holds for all $p$ sufficiently large. We say the the asymptotic relation \eqref{eqar} is an \emph{asymptotic extension} of each of the weighted congruences \eqref{eqtrunc}.
\end{definition}

It turns out that asymptotic relations are fairly common. In the present work we derive two infinite families of asymptotic relations. We pose the following question:

\begin{ques}
\label{quesar}
Determine the set of asymptotic relations.
\end{ques}
\noindent We believe that every weighted congruence admits an asymptotic extension (see Conjecture \ref{conext}), in which case solving Question \ref{quesar} would give a solution to Question \ref{quesMHS'}.

%
%
\subsection{Results}
The present work establishes an algebraic framework for investigating weighted congruences and asymptotic relations. We use two topological rings: a completion of the ring of quasi-symmetric functions, and a projective limit of rings related to the ring $\cA$ used to study finite \mzvs{}.
We discuss a known framework for \mzvs{} in Sec.\ \ref{genmzv}, and we describe our modified framework in Sec.\ \ref{secMHS}.

In Sec.\ \ref{secATs} we prove two results. The first result, stated as Theorem \ref{eART} and in another form as Theorem \ref{ecART}, gives an asymptotic extension of the known congruence
\[
\H(s_1,\ldots,s_k)\equiv(-1)^{s_1+\ldots+s_k}\H(s_k,\ldots,s_1)\mod p,
\]
which holds for all compositions $(s_1,\ldots,s_k)$ and all $p$.
The second result, Theorem \ref{ecADT}, gives asymptotic extensions of a duality result for finite \mzvs{}, proven by Hoffman \cite{Hof04a} and extended by Zhao \cite{Zha08}.
\smallskip


In Sec.\ \ref{seccalc} we use the results of Sec.\ \ref{secATs} to derive two congruences holding modulo high powers of $p$. We provide a script automating our method of computation for the computer algebra system Mathematica.

In Sec.\ \ref{secim} we consider various quantities appearing in arithmetic that can be expressed in terms of weighted multiple harmonic sums. These quantities include values of the $p$-adic zeta function at positive integers.

\subsection{Recent related work}
Several recent results establish homogeneous relations for finite \mzvs. Pilehrood, Pilehrood, and Tauraso \cite{Pil12} showed (among other interesting results) that for all positive integers $a$, $b$, the congruence
\[
\H(\{2\}^a,3,\{2\}^b)\equiv \frac{(-1)^{a+b}(a-b)}{(a+1)(b+1)}{2a+2b+2\choose 2a+1}B_{p-2a-2b-3}\mod p
\]
holds for $p> 2a+2b+3$. 
There are known expressions relating Bernoulli numbers to depth 2 finite \mzvs{} (see \cite{Zha08}). Although it requires some choices, we can use these expressions to tranform the congruence above into the relation
\[
\zeta_{\cA}(\{2\}^a,3,\{2\}^b)+(-1)^{a+b}\frac{2(a-b)}{b+1}\zeta_{\cA}(2a+2,2b+1)=0.
\]
Linebarger and Zhao \cite{Lin13} recently generalized some of the results of \cite{Pil12}.

Recent work of Saito and Wakabayashi concerns finite \mzvs. In \cite{Sai13a} a conjecture of Kaneko is resolved: for all positive integers $k<w$, the congruence
\[
\sum_{\substack{s_1\geq 2,s_2,\ldots,s_k\geq 1\\s_1+\ldots+s_k=w}}\H(s_1,\ldots,s_k)\equiv\lp1+(-1)^k {w-1\choose k-1}\rp\frac{B_{p-w}}{w}\mod p
\]
holds for $p$ sufficiently large. Again using depth 2 finite \mzvs, this can be put in the pleasant form
\[
\sum_{\substack{s_1\geq 2,s_2,\ldots,s_k\geq 1\\s_1+\ldots+s_k=w}}\zeta_{\cA}(s_1,\ldots,s_k)=\frac{k}{w}\zeta_{\cA}(k,w-k)-\frac{1}{w}\zeta_{\cA}(1,w-1).
\]
See also \cite{Sai13} for more results on finite \mzvs.

%
%
\section{Multiple zeta values and Hoffman's algebra}
\label{genmzv}

A \emph{quasi-symmetric} function over $\Q$ is a formal power series of bounded degree in the variables $x_1,x_2,\ldots$, having coefficients in $\Q$, such that the coefficient of $x_{i_1}^{s_1}\ldots x_{i_k}^{s_k}$ is the same as the coefficient of $x_{j_1}^{s_1}\ldots x_{j_k}^{s_k}$ whenever $i_1<\ldots<i_k$ and $j_1<\ldots<j_k$. The set of quasi-symmetric functions is a commutative ring, denoted $\qs$, and it is spanned by the monomial quasi-symmetric functions
\[
M_{\bs}:=\sum_{i_1<\ldots<i_k}x_{i_1}^{s_1}\ldots x_{i_k}^{s_k},
\]
$\bs=(s_1,\ldots,s_k)$ a composition. Our results are the easiest to express using notation introduced by Hoffman \cite{Hof97} to study multiple zeta values. 


\subsection{Hoffman's algebra}
\label{subsecalg}
We consider polynomial algebra
\[
\h:=\Q\langle x,y\rangle
\]
in non-commuting variables $x$ and $y$, and the two subalgebras 
\[
\h^1:=\Q+\h y,\gap \h^0:=\Q+x\h y.
\]
A basis of $\h^1$ consists of words in the non-commuting symbols $z_1,z_2,\ldots$, where $z_n=x^{n-1}y$.  A basis for $\h^0$ consists of those words $z_{s_1}\ldots z_{s_k}$ with $s_1\geq 2$. There is a $\Q$-linear map $\zeta:\h^0\to\R$, taking $z_{s_1}\ldots z_{s_k}\in\h^0$ to the multiple zeta value $\zeta(s_1,\ldots,s_k)\in\R$. The kernel of $\zeta$ is then identified with the space of relations among \mzvs{}.

The non-commutative ring $\h^1$ can be given a commutative multiplication $*$, called the \emph{harmonic} (or \emph{stuffle}) product, which is defined recursively. We set
\[
1*\alpha=\alpha* 1=\alpha
\]
for every $\alpha\in\h^1$. For $k_1,k_2\geq 1$ integers and $\alpha_1,\alpha_2\in\h^1$, we set
\[
(z_{k_1}\alpha_1)*(z_{k_2}\alpha_2)=z_{k_1}(\alpha_1*(z_{k_2}\alpha_2))+z_{k_2}((z_{k_1}\alpha_1)* \alpha_2)+z_{k_1+k_2}(\alpha_1*\alpha_2).
\]
This restricts to give a commutative multiplication on $\h^0$. The harmonic product is constructed to reflect multiplication of nested sums over integers, and we have
\begin{equation*}
\label{estuffhom}
\zeta(\alpha_1*\alpha_2)=\zeta(\alpha_1)\zeta(\alpha_2)
\end{equation*}
for all $\alpha_1,\alpha_2\in\h^0$. We write $\h^1_*$ (resp.\ $\h^0_*$) for the set $\h^1$ (resp.\ $\h^0$) viewed as a commutative ring with multiplication given by $*$, so that $\zeta:\h^0_*\to\R$ is a ring homomorphism.

There is an isomorphism of commutative rings $\h^1_*\cong\qs$ taking $z_{\bs}\in\h^1$ to the monomial quasi-symmetric funtion $M_{\bs}\in\qs$.
The advantage of using the notation of non-commutative polynomials in $x$ and $y$ is that the multiple zeta value $\zeta(\bs)$ can be expressed as an iterated integral, and the specific form of the integral can be read off of the expansion of $z_{\bs}$ as a product of factors $x$ and $y$. This notation will also be convenient for us in the study of asymptotic relations, though for different reasons.

\subsection{Relations among multiple zeta values}
\label{ssrel}
Question \ref{quesmzv} asks for a description of the kernel of $\zeta:\h^0\to\R$. There are many known methods to produce elements of $\ker(\zeta)$. We mention three of them here.

\begin{wenumerate}
\item Let $\tau:\h\to\h$ be the concatenation-reversing automorphism which interchanges $x$ and $y$. Then $\tau$ restricts to an automorphism of $\h^0$. The Duality Theorem for \mzvs, conjectured by Hoffman \cite{Hof92} and a consequence of the iterated integral representation, states that $\zeta(\alpha)=\zeta(\tau(\alpha))$ for all $\alpha\in\h^0$, i.e., $\tau(\alpha)-\alpha\in\ker(\zeta)$.

\item Let $D:\h\to\h$ be the derivation (for the concatenation product) given on generators by $D(x)=0$, $D(y)=xy$. By restriction $D$ gives a derivation of the subalgebra $\h^0$. Set $\overline{D}:=\tau D\tau$ (where $\tau$ is the map defined above). Hoffman (\cite{Hof92} Theorem 5.1, \cite{Hof97} Theorem 6.3) showed that 
$\overline{D}(\alpha)-D(\alpha)\in\ker(\zeta)$
for all $\alpha\in\h^0$. This was generalized by Ohno \cite{Ohn99} and further generalized by Ihara, Kaneko, and Zagier \cite{Iha06}.

\item  Expanding a product of iterated integrals gives a second way to express a product of \mzvs{} as a linear combination of \mzvs{}, which is in general distinct from the harmonic product. Equality of the two product representations gives the \emph{double shuffle relations} (see \cite{Iha06} for a discussion). This can be generalized to allow $\alpha_1,\alpha_2\in\h^1$ using a renormalization, leading to the \emph{extended double shuffle relations}, which conjecturally generate $\ker(\zeta)$.
\end{wenumerate}



%
%
\section{Algebraic framework for classifying\\ weighted congruences and asymptotic relations}
\label{secMHS}
In this section we describe a new algebraic framework for studying weighted congruences and asymptotic relations.
\subsection{Completion of $\h^1$}
 As in the case of \mzvs{}, we consider the non-commutative algebra $\h$ with subalgebra $\h^1$.
The harmonic product $*$ gives a commutative multiplication on $\h^1$.
We set $z_n:=x^{n-1}y$ (which has degree $\deg(z_n)=n$), so that a basis of $\h^1$ consists of the elements
\[
z_{\bs}:=z_{s_1}\ldots z_{s_k},
\]
where $\bs=(s_1,\ldots,s_k)$ is a composition. Our framework uses a completion of $\h^1$.
\begin{definition}
Let $\hh^1$ denote the completion of $\h^1$ with respect the grading by degree. We may view an element of $\hh^1$ as a formal infinite sum
\begin{equation*}
\label{edefalpha}
\alpha=\sum_{\bs}\alpha_{\bs}\,z_{\bs}
\end{equation*}
in the non-commuting variables $z_1,z_2,\ldots$, where the summation is taken over all compositions $\bs=(s_1,\ldots,s_k)$, and the coefficients $\alpha_{\bs}$ are in $\Q$. The elements $z_{\bs}$ span a desnse subspace of $\hh^1$.

For $n\geq 0$, we let $\I_n\subset\hh^1$ be the set of elements of degree $n$ and larger:
\[
\I_n:=\left\{\sum_{\bs}\alpha_{\bs}\,z_{\bs}\in\hh^1:\alpha_{\bs}=0\text{ whenever }w(\bs)<n\right\}.
\]
These sets are a neighborhood basis of 0 for the topology on $\hh^1$.
The harmonic product $*$ gives a commutative multiplication on $\hh^1$ which is continuous. We write $\hhqs$ for the set $\hh^1$, considered as a commutative topological ring with multiplication given by $*$.
\end{definition}

The ring $\hhqs$ is isomorphic to the completion of the ring of quasi-symmetric functions over $\Q$ with respect to the grading by degree (the isomorphism sends $z_{\bs}$ to $M_{\bs}$).
The open sets $\I_n$ are ideals of $\hhqs$.

%
%
\subsection{The ring of asymptotic numbers}
We next define a ring $\as$ that is a kind of $p$-adic analogue of the ring $\cA$ used in the study of finite \mzvs. In studying mod $p^n$ congruences, one is led to consider
\[
\cA_n:=\frac{\prod_p \Z/(p^n)}{\bigoplus_p\Z/(p^n)},\gap n\geq 1.
\]
For each $n\geq 2$ there is a canonical surjection $\varphi_n:\cA_n\to\cA_{n-1}$ coming from reduction modulo $p^{n-1}$. The following definition is new.
\begin{definition}
We define $\as$ be the projective limit of the system of rings $\{\cA_n\}$. An element of $\as$ mau be viewed as an element of the product 
\[
(r_n)\in\prod_{n=1}^\infty \cA_n
\]
such that $\varphi_n(r_n)=r_{n-1}$ for all $n\geq 2$.  There is a canonical projection map $\pi_n:\as\to\cA_n$, sending $(r_n)\in\as$ to $r_n\in\cA_n$. We put the discrete topology on each $\cA_n$ and the projective limit topology on $\as$: the sets $\pi_n^{-1}(0)\subset\as$ form a neighborhood basis of 0.
\end{definition}

As a projective limit of discrete rings, $\as$ is complete. It is not locally compact, however, because each basic open subgroup $\pi^{-1}_n(0)$ contains the open subgroup $\pi_{n+1}^{-1}(0)$ of infinite index.

Given an element $a_p\in\Z_p$ for all but finitely many primes $p$ (here $\Z_p$ is the ring of $p$-adic integers), we get a corresponding element $r_n:=(a_p+p^n\Z)\in \cA_n$ for each $n$. These elements satisfy $\varphi_n(r_n)=r_{n-1}$ for $n\geq 2$, so they determine an element $\as$, which we denote $[a_p]$. It can be checked that every element of $\as$ arises this way, and that $[a_p]=[b_p]$ if and only if $v_p(a_p-b_p)\to\infty$ as $p\to\infty$.


%
%
\subsection{The weighted finite multiple zeta function} 
Next we construct a weighted analogue of the finite multiple zeta function. The following definition is new:
\begin{definition}
For each $n\geq 1$, define a ring homomorphism
\[
\zt_{n}:\hhqs\to\cA_n,
\]
\[
\sum_{\bs}\alpha_{\bs}\,z_{\bs}\mapsto\left[\sum_{w(\bs)<n}\alpha_{\bs}\,p^{w(\bs)}\H(\bs)+p^n\Z\right].
\]
The maps $\zt_{n}$ are compatible with the surjections $\varphi_n:\cA_{n}\to\cA_{n-1}$, so we define
\[
\z:\hhqs \to\as
\]
to be the unique map such that $\pi_n\circ\z=\zt_n$ for all $n\geq 1$. Then for all compositions $\bs$
\[
\z(\bs):=\z(z_{\bs})=\big[p^{w(\bs)}\H(\bs)\big]\in\as
\]
is called a \emph{weighted finite multiple zeta value} (or \emph{weighted finite MZV}). We call $\z$ the \emph{weighted finite multiple zeta function}.
\end{definition}
The map $\zt_n$ is continuous because its kernel contains the open ideal $\I_n\subset\hhqs$. This implies $\z$ is also continuous.
By construction we have
\begin{align*}
\sum_{\bs}\alpha_{\bs} \,z_{\bs}\in\J{n} \gap\Longleftrightarrow&\gap \sum_{w(\bs)<n}\alpha_{\bs} \,p^{w(\bs)}\H(\bs)\equiv 0\mod p^n\text{ for large $p$},\\
\sum_{\bs}\alpha_{\bs} \,z_{\bs}\in\Ji \gap\Longleftrightarrow&\gap \sum_{w(\bs)<n}\alpha_{\bs} \,p^{w(\bs)}\H(\bs)\text{ is an asymptotic relation}.
\end{align*}
Our study of weighted congruences and asymptotic relations amounts to the study of the closed ideals $\J{n}$ and $\Ji$. We may view $\ker(\zt_n)/\I_n$ as the space of weighted congruences holding modulo $p^n$, and $\Ji$ as the space of asymptotic relations. An asymptotic extension of $\alpha\in\J{n}$ is then an element $\alpha'\in\Ji$ such that $\alpha'\equiv\alpha\mod\I_n$.
We make the following conjecture, which is equivalent to the statement that every weighted congruence admits an asymptotic extension.

\begin{conj}
\label{conext}
We have an equality of ideals
\[
\J{n}=\Ji+\I_n.
\]
\end{conj}
\noindent This conjecture would imply the following consequence for homogeneous relations among finite \mzvs{}: suppose $n$ is a positive integer and $\alpha_{\bs}\in\Q$ for each composition $\bs$ of weight $n$, such that the relation for finite \mzvs{}
\[
\sum_{w(\bs)=n}\alpha_{\bs}\,\zeta_{\cA}(\bs)=0
\]
holds. Then there are coefficients $\beta_{\bs}\in\Q$ for each composition $\bs$ of weight $n+1$ such that for all sufficiently large $p$, we have the weighted congruence
\[
\sum_{w(\bs)=n}\alpha_{\bs}\,\H(\bs)\equiv p\sum_{w(\bs)=n+1}\beta_{\bs}\,\H(\bs)\mod p^2.
\]


In many instances values of the weighted finite multiple zeta function can be computed using convergent $p$-adic series identities. The following result is not sharp, but will suffice for our needs.
\begin{proposition}\label{propconv}
Let
\[
\alpha=\sum_{\bs}\alpha_{\bs}\,z_{\bs}\in\hh^1,
\]
and suppose there is a positive integer $k$ such that the denominator of each $\alpha_{\bs}$ is $k$-th power free. Then
\[
a_p:=\sum_{\bs}\alpha_{\bs}\,p^{w(\bs)}\H(\bs)
\]
is $p$-adically convergent and lies in $\Z_p$ for all sufficiently large $p$, and $\z(\alpha)=[a_p]$.
\end{proposition}

\begin{proof}
The hypothesis that the denominators of the $\alpha_{\bs}$ are $k$-th power free implies that $\alpha_{\bs} p^{w(\bs)}\H(\bs)\to 0$ $p$-adically as $w(\bs)\to\infty$ for all $p$, so that $a_p$ converges to an element of $\Q_p$. We will have $a_p\in\Z_p$ for all $p$ not dividing the denominator of any $\alpha_{\bs}$ for $w(\bs)<k$. Finally,
\[
a_p\equiv \sum_{w(\bs)<n}\alpha_{\bs}\,p^{w(\bs)}\H(\bs)\mod p^n
\]
for all $p$ except the finitely many dividing the denominator of some $\alpha_{\bs}$ for $w(\bs)<n+k$. This completes the proof.
\end{proof}

%
%
\subsection{Generation of the space of asymptotic relations}
We would like to produce a generating set for the ideal $\Ji$. One difficulty is that from an algebraic viewpoint, this ideal is not very nice: it may not even be countably generated. However, as the kernel of a continuous map, $\Ji$ is closed in $\hhqs$. As we now show, every closed ideal of $\hhqs$ is the closure of a particular kind of countably generated ideal.

\begin{proposition}
Let $I\subset\hhqs$ be a closed ideal. Then there exist $\alpha_1,\alpha_2,\ldots\in I$, with $\alpha_n\to 0$, such that $I$ is the closure of the ideal generated by $\alpha_1,\alpha_2,\ldots$.
\end{proposition}
\begin{proof}
We claim that there exists a sequence of finite subsets $S_0, S_1,\ldots\subset I$ such that
\[
S_n\subset\I_n\sg\text{ and }\sg(S_0)+(S_1)+\ldots+(S_{n-1})+\I_n=I+\I_n
\]
for all $n\geq 0$. First we show the claim implies the desired result. Take the sequence $\alpha_1,\alpha_2,\ldots$ to be the elements of $S_1$, followed by the elements of $S_2$, and so on. The condition $S_n\subset\I_n$ implies $\alpha_n\to0$. By contruction, for all $n\geq1$ we have
\begin{equation}
\label{eqid}
(\alpha_1,\alpha_2,\ldots)\subset I\subset (\alpha_1,\alpha_2,\ldots)+\I_n.
\end{equation}
The ideals $\I_n$ are a neighborhood basis of 0, so we have
\[
\bigcap_{n=1}^\infty\bigg[ (\alpha_1,\alpha_2,\ldots)+\I_n\bigg]=\overline{(\alpha_1,\alpha_2,\ldots)}.
\]
Intersecting \eqref{eqid} over all $n$ gives
\[
(\alpha_1,\alpha_2,\ldots)\subset I\subset \overline{(\alpha_1,\alpha_2,\ldots)}.
\]
The rightmost containment is then equality because $I$ is closed.

To prove the claim, we construct $S_n$ recursively. We start with $S_0=\varnothing$. Suppose $S_0,\ldots,S_{n-1}$ have been constructed. The ring $\hhqs/\I_{n+1}$ is Noetherian (it is in fact finite dimensional as a vector space), so we may find finitely many elements $\beta_1,\ldots,\beta_k\in I$ whose images in $\hhqs/\I_{n+1}$ generate the ideal $(I+\I_{n+1})/\I_{n+1}$. By hypothesis $(S_0)+\ldots+(S_{n-1})+\I_{n}=I+\I_{n}$, so we may find $\gamma_j\in\I_n$ so that
\[
\beta_j-\gamma_j\in (S_0)+\ldots+(S_{n-1}).
\]
One now checks that $S_n:=\{\gamma_1,\ldots,\gamma_k\}$ has the desired properties.
\end{proof}

 This enables us to formulate a more precise version of Question \ref{quesar}.
\begin{ques4'}
\label{questgen}
Produce a sequence $\alpha_1,\alpha_2,\ldots\in\Ji$, with $\alpha_n\to0$, such that
\[
\Ji=\overline{(\alpha_1,\alpha_2,\ldots)}.
\]
\end{ques4'}

In \cite{Ros12b} we consider the closed subalgebra $\A\subset\hhqs$ of symmetric functions. We produce a sequence of elements of $\Ji\cap\A$ that, conditionally on a conjecture concerning Bernoulli numbers, generate an ideal with closure
$\Ji\cap\A$. This gives a conjectural answer to Question 4$'$ for the case of symmetric functions. We also show that a suitable modification of Conjecture 1, asserting the existence of an asymptotic extension of every weighted congruence, is true for symmetric functions.

\section{The asymptotic reversal and duality theorems}
\label{secATs}
In this section we prove two results giving asymptotic relations extending mod $p$ and mod $p^2$ congruences in the literature. Both results have the following form: we consider a pair $\CCC$ and $\DDD$ of continuous linear automorphisms  of $\hh^1$, with $\CCC$ the exponential of a topologically nilpotent derivation and $\DDD^2=\id$. The theorem statements are then that for all $\alpha\in\hh^1$,
\[
\CCC(\alpha)-\DDD(\alpha)\in\Ji.
\]

\subsection{The asymptotic reversal theorem}
There is an involution on the set of compositions, taking a composition to its reversal
\[
\overline{(s_1,\ldots,s_k)}:=(s_k,\ldots s_1).
\]
It is known (\cite{Hof04a} Theorem 4.5, \cite{Zha08} Lemma 3.3) that the homogeneous congruence
\begin{equation}
\label{eqrev}
\H(\bs)\equiv (-1)^{w(\bs)}\H(\overline{\bs})\pmod{p}
\end{equation}
holds for all primes $p$ and compositions $\bs$. This implies that for $w(\bs)=n$,
\[
z_{\bs}+(-1)^{n+1}z_{\overline{\bs}}\in\J{n+1}.
\]
The following theorem produces an asmptotic extension of this congruence.

\begin{theorem}[Asymptotic Reversal Theorem]
\label{eART}
Let $\bs=(s_1,\ldots,s_k)$ be a composition. For all primes $p$ we have a convergent $p$-adic series equality
\begin{equation*}
\label{eqart}
\H(\bs)=(-1)^{w(\bs)}\sum_{r_1,\ldots,r_k\geq 0}\left[\prod_{j=1}^k{s_j+r_j-1\choose r_j}\right]p^{r_1+\ldots+r_k}\H\big(s_k+r_k,\ldots,s_1+r_1\big).
\end{equation*}
\end{theorem}
Reducing this equation modulo $p$ yields \eqref{eqrev}.
\begin{proof}
For $1\leq m\leq p-1$ an integer, the binomial theorem gives the $p$-adically convergent identity
\[
\frac{1}{(p-m)^s}=\frac{(-1)^s}{m^s}\left(1-\frac{p}{m}\right)^{-s}=(-1)^s\sum_{r\geq0}{s+r-1\choose r}\frac{p^r}{m^{s+r}}.
\]
We then make the substitutions $n_i\leftrightarrow p-m_i$ in the definition of the multiple harmonic sum, giving
\begin{align*}
\H(\bs)&=\sum_{p-1\geq n_1>\ldots>n_k\geq 1}\frac{1}{n_1^{s_1}\ldots n_k^{s_k}}=\sum_{p-1\geq m_k>\ldots>m_1\geq 1}\frac{1}{(p-m_1)^{s_1}\ldots (p-m_k)^{s_k}}\\[10pt]
&=(-1)^{w(\bs)}\sum_{p-1\geq m_k>\ldots>m_1\geq 1}\,\,\,\,\sum_{r_1,\ldots,r_k\geq 0}\left[\prod_{j=1}^k{s_j+r_j-1\choose r}p^{r_j}m_j^{-s_j-r_j}\right]\\[10pt]
&=(-1)^{w(\bs)}\sum_{r_1,\ldots,r_k\geq 0}\left[\prod_{j=1}^k{s_j+r_j-1\choose r_j}\right]p^{r_1+\ldots+r_k}\H(s_k+r_k,\ldots,s_1+r_1)\\
\end{align*}
\end{proof}
In our algebraic framework, we can formulate a version of this theorem in terms of two automorphisms of $\hh^1$.
\begin{definition}
Let $\DR:\hh^1\to\hh^1$ be the continuous linear map
\[
\sum_{\bs}\alpha_{\bs}\,z_{\bs}\sg\longmapsto\sg\sum_{\bs}(-1)^{w(\bs)}\alpha_{\bs} z_{\overline{\bs}},
\]
which preserves the product $*$ and reverses the concatenation product. Let $\CR:\hh\to\hh$ be the continuous linear map given by
\begin{align*}
x\,\mapsto (1-x)^{-1}x\sg=\sg &x+x^2+x^3+\ldots\\
y\,\mapsto (1-x)^{-1}y\sg=\sg &y+xy+x^2y+\ldots,
\end{align*}
extended by linearity, continuity, and multiplicativity to be a homomorphism for the concatenation product.
This restricts to an automorphism of $\hh^1$, also called $\CR$.
\end{definition}

In terms of $\DR$ and $\CR$, Theorem \ref{eART} has the following concise form:

\begin{theorem}[Asymptotic Reversal Theorem, concise form]
\label{ecART}
For all $\alpha\in\hh^1$:
\[
\CR(\alpha)-\DR(\alpha)\in\Ji.
\]
\end{theorem}
\begin{proof}
By linearity and continuity, it suffices to consider $\alpha=z_{\bs}$. We compute
\begin{align*}
\CR(z_{\bs})&=\prod_{j=1}^k\CR(x^{s_j-1}y)=\prod_{j=1}^k (1-x)^{-s_j}x^{s_j-1}y\\
&=\prod_{j=1}^k\,\,\sum_{b_j\geq 0} {s_j+b_j-1\choose b_j}x^{s_j-1+b}y\\
&=\sum_{b_1,\ldots,b_k\geq 0}\left[\prod_{j=1}^k{s_j+b_j-1\choose b_j}\right]z_{s_1+b_1}\ldots z_{s_k+b_k}
\end{align*}
The result now follows from Theorem \ref{eART}.
\end{proof}
It is shown in \cite{Iha06} (Proposition 7) that $\CR$ is the exponential of the derivation $\od$ (for the concatenation product) given on generators
\[
x\mapsto x^2,\gap y\mapsto xy.
\]
The automorphisms $\DR$ and $\CR$ generate an infinite dihedral group inside the group of continuous linear automorphisms of $\hh^1$:
\[
\DR^2=(\DR\CR\DR\CR)^2=1,\gap\#\left\langle\CR\right\rangle=\infty.
\]




\subsection{The asymptotic duality theorem}

Our next result is expressed in terms of two automorphisms of $\hh^1$.
\begin{definition}
Let $\DD:\hh\to\hh$ be the continuous linear map which is a ring homomorphism (for the concatenation product) given on generators by
\[
x\mapsto x+y,\gap y\mapsto -y.
\]
This restricts to a linear automorphism of $\hh^1\to\hh^1$, which we also denote $\DD$.
Let $\CD:\hh^1\to \hh^1$ be the continuous linear automorphism
\[
\alpha\mapsto (1+y)\left(\frac{1}{1+y}*\alpha\right).
\]
\end{definition}

Our results extends congruences modulo $p$ and $p^2$ in the literature. To express these congruences, we write $\H:\h^1\to\Q$ for the linear map given on basis elements by
\[
\H(z_\bs)=\H(\bs).
\]
Then $\DD$ restricts to an automorphism of $\h^1$, and Hoffman \cite{Hof04a} (Theorem 4.7) shows that for all compositions $\bs$, the congruence
\[
\H(\DD(z_\bs))\equiv\H(z_{\bs}) \mod p
\]
holds for all primes $p$. This was extended by Zhao \cite{Zha08} (Theorem 2.11), who proves a result equivalent to the congruence
\[
\H(\DD(z_\bs))\equiv \H(z_{\bs})+p\H(1\sqcup\bs)\mod p^2,
\]
where $1\sqcup(s_1,\ldots,s_k)$ is the composition $(1,s_1,\ldots,s_k)$.

\begin{theorem}[Asymptotic Duality Theorem]
\label{ecADT}
For all $\alpha\in\hh^1$, we have
\[
\CD(\alpha)-\DD(\alpha)\in\Ji.
\]
\end{theorem}
\begin{proof}
By linearity, it suffices to consider $\alpha=z_{\bs}$. First we show that, for $n$ a fixed non-negative integer and $p$ a prime, ${p\choose n}$ can be expressed as a sum involving weighted multiple harmmonic sums. We have
\begin{align*}
{p\choose n}&=\frac{p(p-1)\ldots(p-n+1)}{n!}\\
&=(-1)^{n+1}\frac{p}{n}\left(1-\frac{p}{1}\right)\left(1-\frac{p}{2}\right)\ldots \left(1-\frac{p}{n-1}\right)\\
&=(-1)^{n+1}\frac{p}{n}\sum_{j\geq 0}(-1)^jp^jH_{n-1}(\{1\}^j).
\end{align*}
A result of Hoffman (\cite{Hof04a}, Theorem 4.2 and the proof of Theorem 4.6) shows that for any composition $\bs$,
\[
\H(\DD(z_\bs))=\sum_{n=1}^p (-1)^{n+1}{p\choose n}\H(\bs)
\]
Multiplying through by $p^{w(\bs)}=p^{w(\bs^*)}$ and using our expression for ${p\choose n}$, we have
\begin{gather*}
\z(\DD(z_\bs))=p^{w(\bs^*)}\H(\DD(z_\bs))=\sum_{n=1}^{p}{p\choose n}(-1)^{n+1} p^{w(\bs)}H_{n-1}(\bs)\\
=p^{w(\bs)}\H(\bs)+\sum_{n=1}^{p-1}\frac{p}{n}p^{w(\bs)}\left(\sum_{j\geq 0}(-1)^jp^jH_{n-1}(\{1\}^j)\right)H_{n-1}(\bs)\\
=\z\lp z_{\bs}+y\lp\frac{1}{1+y}*z_\bs\rp\rp.
\end{gather*}
The above computation is linear and continuous is $z_{\bs}$, we for all $\alpha\in\hh^1$ we have
\begin{equation}\label{bd}
\DD(\alpha)-\alpha-y\lp\frac{1}{1+y}*\alpha\rp\in\Ji.
\end{equation}

We observe the asymptotic relation
\[
\sum_{n\geq 1}(-1)^{n+1}p^n\H(\{1\}^n)=0,
\]
which follows for $p\geq 3$ from \cite{Ros12a} (Proposition 2.1, with $n=p-1$, $j=0$). This implies that $y-y^2+y^3-\ldots=\frac{y}{1+y}\in\Ji$. Using this, we can put \eqref{bd} in the desired form. This completes the proof.
\end{proof}

It is shown in \cite{Iha06} (Proposition 6) the $\CD$ is the exponential of the derivation (for the concatenation product) given on generators by
\[
x\mapsto 0,\gap y\mapsto -\sum_{n\geq 1}\frac{x^n y+yx^{n-1}y}{n}.
\]
We also have $\DD^2=\id$.


\section{Calculations in low weight}
\label{seccalc}
In the previous section we gave two methods to produce elements of $\Ji$. In this section we show how to use these elements to write down various weighted congruences. The techniques can be adapted to accomodate additional asymptotic relations that may become known in the future. A script for performing these and other computations in computer algebra system Mathematica can be found in \cite{Mathematica}.

%
%
\subsection{Numerical computation}
For $i=1,2$, we define functions $f_i:\hh^1\to\hh^1$ by
\[
f_i(\alpha):=\CCC_i(\alpha)-\DDD_i(\alpha).
\]
Each $f_i$ is $\Q$-linear, maps each $\I_n$ into itself, and the results of Sec.\ \ref{secATs} imply $f_i(\alpha)\in\Ji$ for all $\alpha\in\hh^1$.

We will work in the quotient $\hh^1/\I_n$, which is finite dimensional. For computational purposes we must choose an ordered basis of $\hh^1/\I_n$, so we choose the elements $z_{\bs}$ as $\bs$ ranges over the compositions of weight less than $n$. We find it convenient to order the compositions determining our basis elements first by weight (lowest weight comes first), then by lexicographic order (smaller numbers come first).

We have that 
\begin{equation}\label{eqf}
\alpha*f_i(\beta)\in\Ji
\end{equation}
for $i=1,2$ and all $\alpha,\beta$. Because we work in $\hh^1/\I_n$ and \eqref{eqf} is bilinear in $\alpha$ and $\beta$, it suffices to take $\alpha$, $\beta$ in our chosen basis for $\hh^1/\I_n$. This gives us a finite collection of elements.

\begin{definition}
Let $M_n$ be the matrix whose rows are indexed by triples  $(\bs_1,\bs_2,i)$ with $w(\bs_1)+w(\bs_2)<n$ and $i\in\{0,1\}$, whose columns are indexed by compositions $\bs$ with $w(\bs)<n$,  and whose entry in row  $(\bs_1,\bs_2,i)$, column $\bs$ is the coefficient of $z_{\bs}$ in
\[
z_{\bs_1}*f_i(z_{\bs_2})\in\hh^1/\I_n.
\]
The entries of $M_n$ are rational numbers, and we identify the row span of $M_n$ with the space weighted congruence holding modulo $p^n$ which can be derived from the results of Sec.\ \ref{secATs}. We say a compositions $\bs$ is \emph{essential} if $\bs$ is not a pivot column of $M_n$ for any (or equivalently all) $n>w(\bs)$.
\end{definition}
This notion of essential depends on our choice of ordering for compositions. It also depends on the source of asymptotic relations (here we consider only those asymptotic relations coming from Sec.\ \ref{secATs}). The set
\[
\{\zt_n(z_{\bs}):w(\bs)<n,\text{ $\bs$ essential}\}
\]
spans the image of $\zt_n$.
For example, when $n=7$ we find the essential compositions are $(2,1)$, $(4,1)$, and $(4,1,1)$.


\subsection{Example computations}
Given any element $\alpha\in\hh^1$ and any $n\geq 1$, one may compute the matrix $M_n$ described above and perform row reduction to find a unique element
\[
\p(\alpha)=\sum_{\substack{w(\bs)<n\\\bs\text{ essential}}}\alpha'_{\bs}\,z_{\bs}
\]
which is congruent modulo the row span of $M_n$ to the image of $\alpha$ in $\hh^1/\I_n$. This will mean $\alpha-\p(\alpha)\in\J{n}$. The map $\alpha\mapsto\p(\alpha)$ is linear. Note that we have ordered compositions by weight, so $\p$ takes $\I_m$ to $\I_m$ for all $m\geq 1$.

\subsubsection{Extending known congruences}\label{sswolst} Given $\alpha\in\hh$, we have $\alpha-\p(\alpha)\in\J{n}$, so we get a weighted congruence. As an example we take $\alpha=z_{(1)}$, $n=7$. Numerical computation gives
\[
\p(\alpha)=-\frac{1}{3}z_{(2,1)}+\frac{1}{6}z_{(4,1)}+\frac{1}{9}z_{(4,1,1)},
\]
so the weighted congruence
\[
\H(1)+\frac{1}{3}p^2\H(2,1)-\frac{1}{6}p^4\H(4,1)-\frac{1}{9}p^5\H(4,1,1)\mod p^6
\]
holds for all sufficiently large primes $p$. Reducing modulo $p^6$ gives the congruence \eqref{eqex} stated in the introduction, which is an extension of Wolstenholme's congruence. This calculation can be repeated with $\alpha$ replaced by any element of $\hh^1$.

\subsubsection{Congruences holding modulo high powers}
Given $\alpha_1,\ldots,\alpha_k\in\hh^1$ and $n\geq 1$, we can compute $\p(\alpha_1),\ldots,\p(\alpha_n)$. The $\p(\alpha_i)$ lie in a vector space spanned by the essential compositions of weight less than $n$. This space is not too big, so we may attempt to find a non-trivial linear relation
\[
\sum_{i=1}^k a_i \p(\alpha_i)=0.
\]
If we can find such a relation, we will have
\[
\sum_{i=1}^k a_i\alpha_i\in\Ji+\I_n\subset\J{n},
\]
giving a weighted congruence .
As an example, we set $n=10$ and take the $\alpha_i$ to be $z_{\bs}$ for compositions $\bs$ of weight 5 and depth at most 2: $\alpha_1=z_{(1,4)}$, $\alpha_2=z_{(2,3)}$, $\alpha_3=z_{(3,2)}$, $\alpha_4=z_{(4,1)}$, $\alpha_5=z_{(5)}$. We compute the corresponding elements $\p(\alpha_1),\ldots,\p(\alpha_5)$ and we find they satisfy the linear relation
\[
3\p(\alpha_1)-\p(\alpha_2)-\p(\alpha_3)+3\p(\alpha_4)+2\p(\alpha_5)=0.
\]
This shows that the homogeneous weighted congruence
\[
3\H(1,4)-\H(2,3)-\H(3,2)+3\H(4,1)+2\H(5)\equiv 0\mod p^5
\]
holds for all $p$ sufficiently large.

%
%
\section{The image of $\z$}
\label{secim}
In this section we investigate the image of the weighted finite multiple zeta function $\z$.
Next we start with the following:

\begin{proposition}
\label{propclim}
The image of $\zt_\infty$ is closed in $\as$.
\end{proposition}
We recall that a complete topological ring is called \emph{linearly topologized} if the topology admits a neighborhood basis of ideals. Both topological rings $\hhqs$ and $\as$ considered so far are linearly topologized.
Proposition \ref{propclim} is an immediate consequence of a more general fact:

\begin{lemma}
Let $\varphi:R_1\to R_2$ be a continuous homomorphism of linearly topologized rings, and suppose $R_1$ is the projective limit of a sequence of Artinian rings. Then the image of $\varphi$ is closed.
\end{lemma}
\begin{proof}
Let $R_1\supset J_1\supset J_2\supset\ldots$ and $\{I_\alpha\subset R_2\}$ be ideals of forming neghborhood bases of 0.
Then for each $\alpha$
\[
J_\alpha:=\varphi^{-1}(I_\alpha)
\]
is an open ideal of $R_1$, so contains some $I_n$, say $I_{n(\alpha)}$. We get maps 
\[
\varphi_\alpha:R_1/I_{n(\alpha)}\to R_2/I_\alpha.
\]

Suppose $r\in R_2$ is in the closure of the image of $\varphi$. For each $\alpha$, define
\[
S_\alpha:=\varphi_\alpha^{-1}(r+I_\alpha)\subset R_1.
\]
This is a coset of the ideal $I_{n(\alpha)}$. The image of $\varphi$ intersects $r+I_\alpha$ nontrivially, so $S_\alpha$ is nonempty. For $1\leq n\leq m$, the quotient map $R_1/I_m\to R_1/I_n$ takes $S_m$ into $S_n$. The image of $S_m$ is a coset of an ideal in $R_1/I_n$, and by hypothesis the ring $R_1/I_{n}$ is Artinian so satisfies the descending chain condition on cosets of ideals. This means the inverse system $S_1\leftarrow S_2\leftarrow\ldots$ satisfies the Mittag-Leffler condition, so since each $S_n$ is non-empty, there are elements $r'$ of $R_1$ which map into every $S_n$. Any such $r'$ will map to $r$ by $\varphi$. 
\end{proof}


The closed subring $\im(\z)\subset\as$ is very small, in the following precise sense: for each integer $n\geq 1$, the map $\zt_n:\hhqs$ to $\cA_{n}$ factors through the countable ring $\hhqs/I_n$. This means the $\pi_n(\im(\z))\subset\cA_n$ is countable. However, $\cA_n$ is the quotient of a countably infinite product of finite sets by a countable ideal, so has continuum cardinality.

%
%
\subsection{Asymptotic representability}
Now we produce certain some in the image of the weighted finite multiple zeta function $\z$. First we give a definition.

\begin{definition}
A collection of elements $a_p\in\Z_p$ is said to be \emph{asymptotically representable by weighted multiple harmonic sums}, or simply \emph{asymptotically representable}, if the corresponding element $[a_p]\in \as$ is in the image of $\z$. In other words, the collection $a_p\in\Z_p$ is asymptotically representable if and only if there are coefficients $\alpha_{\bs}\in\Q$ not depending on $p$ such that for all $n\geq 1$ the congruence
\[
a_p\equiv\sum_{w(\bs)<n}\alpha_{\bs}\,p^{w(\bs)}\H(\bs)\mod p^n
\]
holds for all sufficiently large $p$.
\end{definition}

Our next result concerns a family of sums related to multiple harmonic sums.
\begin{theorem}
For all compositions $\bs=(s_1,\ldots,s_k)$ and all positive integers $r$, the generalized weighted multiple harmonic sum
\[
p^{w(\bs)}H_{pr}^{(p)}(\bs):=p^{w(\bs)}\sum_{\substack{pr\geq n_1>\ldots>n_k\geq 1\\p\nmid n_1n_2\ldots n_k}}\frac{1}{n_1^{s_1}\ldots n_k^{s_k}}
\]
is asymptotically representable.
\end{theorem}
\begin{proof}
First we note the convergent $p$-adic expansion
\[
(jp+n)^{-s}=\sum_{i=0}^{\infty}{-s\choose i}j^i\frac{p^i}{n^{s+i}},
\]
which holds for $j\in\Z_p$, $n\in\Z_p^\times$, $s\in\Z$. Writing $n_i=j_ip+m_i$:

\begin{gather*}
p^{w(\bs)}H_{pr}^{(p)}(\bs)\,=\sum_{j_1\geq\ldots\geq j_k=0}^{r-1}\left(\sum_{\substack{p-1\geq m_1,\ldots,m_k\geq 1\\m_i>m_{i+1}\text{ if }j_i=j_{i+1}}}\frac{p^{w(\bs)}}{(j_1p+m_1)^{s_1}\ldots(j_kp+m_k)^{s_k}}\right)\\
=\sum_{j_1\geq\ldots\geq j_k=0}^{r-1}\left(\sum_{\substack{p-1\geq m_1,\ldots,m_k\geq 1\\m_i>m_{i+1}\text{ if }j_i=j_{i+1}}}\prod_{q=1}^{k}\sum_{i_q=0}^{\infty}{-s_q\choose i_q}j_q^{i_q}\frac{p^{s_q+i_q}}{m_q^{s_q+i_q}} \right)\\
=\sum_{\substack{0\leq j_1,\ldots,j_k\leq r-1\\0\leq i_1,\ldots,i_k}}\left(\prod_{q=1}^k{-s_q\choose i_q}j_q^{i_q}\right)\prod_{j=0}^{r-1} p^{w(\bt_j)}\H(\bt_j),\\[5mm]
\hspace{2cm}\text{where }\hspace{1cm}\bt_j:=(s_{a}+i_{a},\ldots,s_{b}+i_{b}),\gap \{a,a+1,\ldots,b\}=\{n:j_n=j\}.
\end{gather*}

The coefficients appearing are all integral, so Proposition \ref{propconv} now implies that
\[
\bigg[p^{w(\bs)}H_{pr}^{(p)}(\bs)\bigg]=\z\left(\sum_{\substack{0\leq j_1,\ldots,j_k\leq r-1\\0\leq i_1,\ldots,i_k}}\left(\prod_{q=1}^k{-s_q\choose i_q}j_q^{i_q}\right) z_{\bt_0}*\ldots*z_{\bt_{r-1}}\right).
\]
\end{proof}

Our next result concerns binomial coefficients of a certain shape.
\begin{theorem}
For all integers $k$, $r$, with $r\geq 0$, the binomial coefficient
\[
\displaystyle {kp\choose rp}
\]
is asymptotically representable.
\end{theorem}
\begin{proof}
For $p$ odd, we write
\begin{gather*}
{kp\choose rp}=\frac{(kp)(kp-1)\ldots ((k-r)p+1)}{(rp)(rp-1)\ldots(1)}\\
={k\choose r}\prod_{\substack{j=1\\p\nmid j}}^{rp-1}\left(1-\frac{kp}{j}\right)={k\choose r}\sum_{n\geq 0}(-k)^n p^n H_{rp}^{(p)}(\{1\}^n).
\end{gather*}
The result now follows from the asymptotic representability of $p^nH_{rp}^{(p)}(\{1\}^n)$ and the fact that $\im(\z)$ is closed.

\end{proof}

Our final result gives an asymptotic representation of values of the $p$-adic zeta function, which we now describe. Values taken by the Riemann zeta function at negative integers are rational, and can be expressed in terms of Bernoulli numbers. More generally for any Dirichlet character $\chi$ values of the $L$-function $L(s,\chi)$ at negative integers lie in the field generated over $\Q$ by the values of $\chi$, so in particular are algebraic. These values can be expressed in terms of generalized Bernoulli numbers.

Let $p$ be an odd prime, $\Qb$ an algebraic closure of $\Q$, and $\C_p$ the completion of an agebraic closure of $\Q_p$ (the field of $p$-adic numbers).  We fix once and for all embeddings of $\overline{\Q}$ into $\C$ and into $\C_p$. This allows us to identiy algebraic elements of $\C$ as living in $\C_p$. Kummer's congruences for the generalized Bernoulli numbers imply that for any primitive Dirichlet character $\chi$, the function
\[
\lp1-\chi(p)p^{-s}\rp L(s,\chi)
\]
is $p$-adically continuous when restricted to the negative integers $s$ in a fixed residue class mod $p-1$. For $\chi\neq 1$, the $p$-adic $L$-funciton of Kubota-Leopoldt is the unique continuous function $\Z_p\to\C_p$, $s\mapsto L_p(s,\chi)$, agreeing with $(1-\chi(p)p^{-s})L(s,\chi)$ when $s$ is a negative integer congruent to $1$ mod $p-1$. For $\chi=1$, $L_p(s,\chi)$ is continuous except for a simple pole at $s=1$.

Define the Teichm\"{u}ller character $\omega:\lp\Z/(p)\rp^\times\to\Q_p^\times$ to be the unique group homomorphism such that $\omega_p(n)\equiv n\mod p$ for all $n\in\Z$ not divisible by $p$. The $p$-adic $L$-function $L_p(s,\omega_p^{1-k})$ agrees with $(1-\chi(p)p^{-s})L(s,\chi)$ for negative integers $s$ congruent to $k$ mod $p-1$. For $k\geq 2$, we define the $p$-adic zeta value
\[
\zeta_p(k):=L_p(k,\omega^{1-k})\in\Q_p.
\]
It is worth noting that $\zeta_p(k)$ is not $p$-adically continuous as a function of $k$, but comes from $p-1$ different continous functions, one defined on each residue class mod $p-1$. The vanishing of the odd Bernoulli numbers implies $\zeta_p(2k)=0$ for $k\geq 1$.

\begin{theorem}
For all integers $k\geq 2$, $p^k\zeta_p(k)$ is asymptotically representable.
\end{theorem}

\begin{proof}
 In \cite{Was98} (Theorem 1 with $n=1$), Washington shows that for $r\geq 1$, we have a convergent $p$-adic series identity for a multiple harmonic sums in terms of $p$-adic $\zeta$-values:
\begin{align*}
\H(s)&=-\sum_{k=1}^\infty {-s\choose k}p^k\zeta_p(s+k)\\
&=\sum_{k=s+1}^\infty (-1)^{k+s+1}{k-1\choose s-1}p^{k-s}\zeta_p(k).
\end{align*}
Fix an integer $n\geq 2$. Using the identity above, we can compute:
\begin{align*}
&\sum_{s\geq n-1}\frac{(-1)^{s+n+1}}{n-1}{s-1\choose n-2}B_{s+1-n}\,p^{s}\H(s)\\
=&\sum_{s\geq n-1}\frac{(-1)^{s+n+1}}{n-1}{s-1\choose n-2}B_{s+1-n}\,p^s\sum_{k=s+1}^\infty(-1)^{k+s+1}{k-1\choose s-1}p^{k-s}\zeta_p(k)\\
=&\sum_{k=n}^\infty   \frac{(-1)^{k+n}}{n-1}p^k\zeta_p(k)\sum_{s=n-1}^{k-1}{k-1\choose s-1}{s-1\choose n-2}B_{s+1-n}\\
=&\sum_{k=n}^\infty   \frac{(-1)^{k+n}}{n-1}{k-1\choose n-2}p^k\zeta_p(k)\sum_{s=n-1}^{k-1}{k+1-n\choose s+1-n}B_{s+1-n}\\
=&\sum_{k=n}^\infty   \frac{(-1)^{k+n}}{n-1}{k-1\choose n-2}p^k\zeta_p(k)\sum_{s=0}^{k-n}{k+1-n\choose s}B_{s}.\\
\end{align*}
It is known that
\[
\sum_{s=0}^{k-n}{k+1-n\choose s}B_{s}=\begin{cases}1\text{ if }k=n,\\0\text{ otherwise}\\\end{cases}
\]
(see \cite{Coh07}, p.\ 7), so we obtain
\[
\sum_{s\geq n-1}\frac{(-1)^{s+n+1}}{n-1}{s-1\choose n-2}B_{s+1-n}\,p^{s}\H(s)=p^n\zeta_p(n).
\]
The von Staudt-Clausen theorem implies $B_k$ has squarefree denominator and the factor of $n-1$ in the denominator is constant, so Proposition \ref{propconv} gives
\[
\left[p^n\zeta_p(n)\right]=\z\lp\sum_{s\geq n-1}\frac{(-1)^{s+n+1}}{n-1}{s-1\choose n-2}B_{s+1-n}z_s\rp.
\]
\end{proof}

%

\ack
I thank Jeff Lagarias for numerous helpful discussions and encouragement. This work supported in part by NSF grants  DMS-0943832 and DMS-1101373.

\bibliographystyle{hplain}
\bibliography{jrbiblio}
\end{document}